\documentclass{IEEEtran}

\usepackage{cite}
\usepackage{url}
\usepackage{amsmath,amssymb,amsfonts}
\usepackage{multirow}
\usepackage{amsthm}
\usepackage{algorithmic}
\usepackage{graphicx}
\usepackage{textcomp}
\usepackage{epsfig}
\usepackage{algorithmic}
\usepackage{epstopdf}
\usepackage[ruled]{algorithm}
\usepackage{enumerate}
\usepackage{float}
\usepackage{graphics,graphicx}
\usepackage{subfigure}

\usepackage{array}
\usepackage{mathtools}
\usepackage[justification=centering]{caption}
\allowdisplaybreaks[3]
\usepackage{url}
\DeclareGraphicsExtensions{.png}
\graphicspath{{./Pics/}}
\usepackage{color}
\usepackage{xspace}

\usepackage{hyperref}
\hypersetup{
    colorlinks=true,
    linkcolor=blue,
    filecolor=gray,      
    urlcolor=blue,
    citecolor=blue,
}

\def\BibTeX{{\rm B\kern-.05em{\sc i\kern-.025em b}\kern-.08em
		T\kern-.1667em\lower.7ex\hbox{E}\kern-.125emX}}

\newtheorem{lemma}{Lemma}
\newtheorem{theorem}{Theorem}

\newtheorem{assumption}{Assumption}

\begin{document}
	
	\title{\LARGE \bf A Stochastic Second-Order Proximal Method\\for Distributed Optimization}
	
	\author{Chenyang Qiu, Shanying Zhu, Zichong Ou and Jie Lu\thanks{C. Qiu, Z. Ou and J. Lu are with the School of Information Science and Technology, ShanghaiTech University, 201210 Shanghai, China. Email: {\tt \{qiuchy, ouzch, lujie\}@shanghaitech.edu.cn}.}\thanks{S. Zhu is with the Department of Automation and the Key Laboratory of System Control and Information Processing, Shanghai Jiao Tong University, 200240 Shanghai, China. Email: {\tt shyzhu@sjtu.edu.cn}.}}
	\maketitle

	\begin{abstract}
	In this paper, we propose a distributed stochastic second-order proximal method that enables agents in a network to cooperatively minimize the sum of their local loss functions without any centralized coordination. The proposed algorithm, referred to as St-SoPro, incorporates a decentralized second-order approximation into an augmented Lagrangian function, and then randomly samples the local gradients and Hessian matrices of the agents, so that it is computationally and memory-wise efficient, particularly for large-scale optimization problems. We show that for globally restricted strongly convex problems, the expected optimality error of St-SoPro asymptotically drops below an explicit error bound at a linear rate, and the error bound can be arbitrarily small with proper parameter settings. Simulations over real machine learning datasets demonstrate that St-SoPro outperforms several state-of-the-art distributed stochastic first-order methods in terms of convergence speed as well as computation and communication costs.
	\end{abstract}

\section{Introduction}\label{sec:introduction}
Stochastic optimization algorithms have been flourishing recently due to their appealing efficiency in machine learning \cite{bottou2010large, bottou2018optimization}. In the context of large-scale machine learning, parallel stochastic algorithms are often used to process large datasets\cite{mcmahan2017communication, stich2018local}. However, such methods need a central node to ensure the consistency of the variables from all the nodes, so that the communication burden of the central node becomes the bottleneck that restricts the algorithm performance. 

On the other hand, a collection of distributed optimization algorithms have been proposed over the past decade in order to tackle various network control and resource allocation problems, where agents in a network only communicate with their neighbors and do not rely on any central coordination, eliminating potential communication bottlenecks in the computing infrastructure \cite{NedicDGD2009}. Typical methods include the distributed gradient descent (DGD) method \cite{NedicDGD2009}, the decentralized exact first-order algorithm (EXTRA) \cite{shi2015extra}, and the distributed gradient tracking algorithms \cite{Nedi2017AchievingGC}.

Inheriting the merits of the above two algorithm types, distributed stochastic optimization algorithms have been attracting a lot of recent interest. For smooth, strongly convex optimization problems, \cite{pu2021sharp} develops a distributed stochastic gradient descent (DSGD) method based on DGD, which is shown to achieve the optimal sublinear convergence rate (independent of the network) of a centralized stochastic gradient descent (SGD) method \cite{Rakhlin2012stochastic,Nemirovski2009RobustSA}. The same convergence rate is attained by the exact diffusion method with adaptive step-sizes (EDAS) in \cite{huang2022improving}. In addition, the distributed stochastic gradient tracking (DSGT) method proposed in \cite{pu2021distributed} is guaranteed to linearly converge to a neighborhood of the optimal solution in expectation. For non-convex problems, \cite{yi2022primal} designs a distributed primal-dual SGD algorithm with both fixed step-sizes (DPD-SGD-F) and adaptive step-sizes (DPD-SGD-T), where the former achieves linear convergence to suboptimality under the Polyak-{\L}ojasiewics (PL) condition. 

The aforementioned distributed stochastic algorithms are all constructed upon deterministic \emph{first-order} methods by nature and only use stochastic gradients to evolve. As second-order information often leads to more accurate approximation and accelerates problem solving, we endeavor to develop a \emph{second-order} distributed stochastic optimization algorithm. 

To this end, we choose SoPro \cite{wuxuyang_sopro}, a deterministic distributed second-order proximal algorithm, as the cornerstone. SoPro is developed by virtue of a decentralized second-order approximation of the augmented Lagrangian function in the classic method of multipliers \cite{Boyd2011DistributedOA}, and its convergence performance outperforms that of various distributed first-order methods in the deterministic setting. In this paper, we adapt SoPro to the stochastic scenario. Specifically, instead of letting each agent compute exact local gradient and Hessian matrix determined by all its local data as SoPro does, we allow each agent to update using stochastic approximations of its local gradient and Hessian, which come from two randomly and uniformly chosen batches of samples from its local loss function. Such a stochastic variant of SoPro can significantly enhance the computational and memory efficiency of the agents. We call this algorithm a \emph{stochastic second-order proximal algorithm}, referred to as St-SoPro.

Under the assumptions that the local loss functions are smooth and convex, and that their sum (i.e., the global objective function) is globally restricted strongly convex, we show that our proposed St-SoPro algorithm \emph{linearly} converges to a neighborhood of the optimal solution in expectation over undirected networks. In particular, we provide an explicit upper bound on its ultimate suboptimality, and illustrate that this upper bound can be made arbitrarily small as long as the parameters are properly set. Finally, we validate the superior performance of St-SoPro in comparison with several recent distributed stochastic optimization methods over some real datasets for classification problems arising in machine learning with respect to convergence speed, communication load, computational efficiency, and classification accuracy.
	
The paper is organized as follows. Section~II formulates the optimization problem to be solved, and Section~III describes the proposed St-SoPro algorithm. Section~IV provides the convergence analysis, Section~V presents the numerical results, and Section VI concludes the paper.
	
\textbf{Notation:} For any differentiable function $f: \mathbb{R}^{d} \rightarrow \mathbb{R}$, its gradient at $x \in \mathbb{R}^{d}$ is denoted by $\nabla f(x)$, and if $f$ is twice-differentiable, we use $\nabla^{2} f(x)$ to denote its Hessian matrix at $x$. For any set $S$, $|S|$ represents the cardinality of $S$. In addition, $\otimes$ is the Kronecker product, $\langle\cdot, \cdot\rangle$ is the Euclidean inner product, and $\|\cdot\|$ is the $\ell_{2}$ norm. We use $\mathbf{0}_{d}, \mathbf{1}_{d}, \mathbf{O}_{d}, I_{d}$ to denote the $d$-dimensional all-zero vector, all-one vector, zero matrix, and identity matrix, respectively. 
Also, $\operatorname{diag}(A_1,\ldots,A_n)$ represents the block diagonal matrix whose diagonal blocks are sequentially $A_1,\ldots,A_n$. 
Given a matrix $A \in \mathbb{R}^{d \times d}$, we write $A\succeq \mathbf{O}_{d}$ if it is positive semidefinite and  $A\succ \mathbf{O}_{d}$ if it is positive definite. For any $A\succeq \mathbf{O}_{d}$ and $\mathbf{x} \in \mathbb{R}^{d},\|\mathbf{x}\|_{A}^{2}:=\mathbf{x}^{T} A \mathbf{x}$, $\lambda_{max}(A)$ and $\lambda_{min}(A)$ are the largest and smallest real eigenvalues of $A$, respectively, and $A^{\dagger}$ is $A$ 's pseudoinverse.
	
	\section{Problem Formulation}\label{sec:formulation}
	Consider a set $\mathcal{V}=\{1, \ldots, N\}$ of agents, where the agents are connected through the link set $\mathcal{E} \subseteq$ $\{\{i, j\} \subseteq \mathcal{V} \times \mathcal{V} \mid i \neq j\}$. We model such a network as a connected undirected graph $(\mathcal{V}, \mathcal{E})$, and denote the set of each agent $i$'s neighbors by $\mathcal{N}_{i}=\{j \in \mathcal{V} \mid\{i, j\} \in \mathcal{E}\}$. Suppose each agent $i$ observes a finite number of local samples $\xi_{i,j}\in\mathbb{R}^m$, $j=1,\ldots,\mathcal{C}_i$ that are independent random vectors, and attempts to solve the following optimization problem:
\vspace{-0.15cm} \begin{equation}
		\begin{array}{ll}
			\underset{x_1,\ldots,x_N\in \mathbb{R}^{d}}{\operatorname{minimize}} &\sum_{i \in \mathcal{V}}  f_{i}(x_i)\\
			\text {subject to } & x_1 = x_2 = \dots = x_N.
		\end{array}\label{formulation}
\vspace{-0.15cm} \end{equation}	
In problem~\eqref{formulation}, $f_{i}: \mathbb{R}^{d} \rightarrow \mathbb{R}$ is the local loss function of agent $i$, which is the average of every sample's loss $l_{i,j}(x_i,\xi_{i,j}): \mathbb{R}^{d} \rightarrow \mathbb{R}$ associated with agent $i$, i.e.,
\begin{align*}
f_i(x_i) = \frac{1}{\mathcal{C}_i} \sum_{j =1}^{\mathcal{C}_i} l_{i,j}(x_i, \xi_{i,j}).
\end{align*}
Below, we impose the assumptions on problem~\eqref{formulation}.
	\begin{assumption}\label{asm: 1}
		Problem~\eqref{formulation} satisfies the following:
		\begin{enumerate}
			\item[a)] There exists an optimal solution $x^{\star}\in \mathbb{R}^{d}$ to problem~\eqref{formulation}, and $\sum_{i \in \mathcal{V}} f_{i}(x_i)$ is globally restricted strongly convex with respect to $x^{\star}$ with convexity parameter $m_{\Bar{f}}>0$.
			\item[b)] For any given $\xi_{i,j}$, $l_{i,j}(\cdot, \xi_{i,j})$ is twice continuously differentiable and convex.
			\item[c)] There exists $M_i>0$ such that $l_{i,j}(x, \xi_{i,j})$ is $M_i$-smooth for all $\xi_{i,j}$.
		\end{enumerate}
	\end{assumption}
Assumption~\ref{asm: 1} leads to the following inequalities. For any given $x,y \in \mathbb{R}^{d}$ and $\xi_{i,j}$, $m_i \|x-y\|^2\leq\langle \nabla l_{i,j}(x,\xi_{i,j})-\nabla l_{i,j}(y,\xi_{i,j}), x-y \rangle\leq M_i \|x-y\|^2$ and 
    \begin{align}
    m_{i} I_{d} \preceq \nabla^{2} l_{i,j}(x,\xi_{i,j}) \preceq M_{i} I_{d}\label{hessian inequality}
    \end{align}
for some  $m_i \in [0,M_i]$. Also, the globally restricted strong convexity in Assumption \ref{asm: 1}a) guarantees that the optimal solution $x^{\star}$ is unique.

Problem~\eqref{formulation} requires that the agents reach a consensus while minimizing all the sample losses throughout the network. Indeed, a wide range of real-world problems can be cast into the form of problem~\eqref{formulation}, such as distributed model predictive control \cite{Giselsson2013AcceleratedGM}, distributed spectrum sensing \cite{Bazerque2010DistributedSS}, and logistic regression \cite{bach2014adaptivity}. Under many circumstances, these engineering problems involve huge datasets. Thus, we focus on solving problem~\eqref{formulation} in a fully decentralized and stochastic fashion. Specifically, we only allow each agent to communicate with its neighbors and use a randomly chosen subset of its local samples to compute.

	\section{Stochastic Second-order Proximal Method}\label{sec:algorithm}
	In this section, we develop a distributed stochastic algorithm for solving problem \eqref{formulation} over undirected networks. 
	
	To this end, we first provide a brief review of the distributed (deterministic) second-order proximal algorithm (SoPro) in \cite{wuxuyang_sopro}. Note that problem~\eqref{formulation} is equivalent to
		\vspace{-0.15cm} \begin{equation}
		\begin{array}{ll}
			\underset{\mathbf{x} \in \mathbb{R}^{N d}}{\operatorname{minimize}} & f(\mathbf{x}):=\sum_{i \in \mathcal{V}}  f_{i}(x_i) \\
			\text {subject to } & {W}^{\frac{1}{2}}\mathbf{x} = \mathbf{0}_{Nd},
		\end{array}\label{rereformulation}
	\vspace{-0.15cm} \end{equation}
where $\mathbf{x} = (x_1^{T},\ldots,x_N^{T})^{T}$, $W = P \otimes I_{d} \succeq \mathbf{O}_{N d}$, and
    \begin{align*}
    [P]_{i j}=\left\{\begin{array}{ll}
    \sum_{s \in \mathcal{N}_{i}} p_{i s}, & i=j, \\
    -p_{i j}, & j \in \mathcal{N}_{i}, \\
    0, & \text {otherwise,}
    \end{array} \quad \forall i, j \in \mathcal{V},\right.
    \end{align*}
with $p_{ij} = p_{ji}>0$ $\forall\{i,j\}\in\mathcal{E}$, so that the null space of $P$ is $\operatorname{span}\{\mathbf{1}_{N}\}$. Also, the unique optimal solution of problem \eqref{rereformulation} is $\mathbf{x}^{\star}=((x^{\star})^{T }, \ldots,(x^{\star})^{T })^{T }$. 

The application of the method of multipliers \cite{Boyd2011DistributedOA} to solve \eqref{rereformulation} gives the following: Starting from any $\mathbf{v}^{0}\in\mathbb{R}^{Nd}$,
	\begin{align}
	&\mathbf{x}^{k+1}=\arg \min _{\mathbf{x} \in \mathbb{R}^{N d}} L_{\beta}\left(\mathbf{x}, \mathbf{v}^{k}\right),\label{xup1}\displaybreak[0]\\
  &\mathbf{v}^{k+1}=\mathbf{v}^{k}+\beta W^{\frac{1}{2}} \mathbf{x}^{k+1}, \label{vup1}
	\end{align}
where $\mathbf{x}^{k}=((x_1^k)^T,\ldots,(x_N^k)^T)^T$ and $\mathbf{v}^{k}$ are the primal and dual variables, respectively, and $L_{\beta}(\mathbf{x}, \mathbf{v}): \mathbb{R}^{N d} \times \mathbb{R}^{N d} \rightarrow \mathbb{R}$ is an augmented Lagrangian function given by $L_{\beta}(\mathbf{x}, \mathbf{v})=f(\mathbf{x})+\mathbf{v}^{T} W^{\frac{1}{2}}\mathbf{x}+\frac{\beta}{2}\|W^{\frac{1}{2}}\mathbf{x}\|^{2}$, $\beta>0$. 

Since \eqref{xup1}--\eqref{vup1} cannot be implemented in a distributed way, the SoPro algorithm in \cite{wuxuyang_sopro} introduces a decentralized second-order proximal approximation of $L_{\beta}\left(\mathbf{x}, \mathbf{v}^{k}\right)$ in \eqref{xup1} and applies a variable change to \eqref{vup1}. Specifically, it replaces $L_{\beta}\left(\mathbf{x}, \mathbf{v}^{k}\right)$ with its second-order Taylor's expansion at $\mathbf{x}^k$. Then, it replaces the remaining coupling term $\frac{1}{2}(\mathbf{x}-\mathbf{x}^k)^T\rho W(\mathbf{x}-\mathbf{x}^k)$ in the primal update with $\frac{1}{2}(\mathbf{x}-\mathbf{x}^k)^TD(\mathbf{x}-\mathbf{x}^k)$, where $D=\operatorname{diag}(D_1, \ldots, D_N)$ is a symmetric block diagonal matrix satisfying $\nabla^2 f_i(x)+D_i \succ\mathbf{O}_d$ $\forall x \in \mathbb{R}^d$ $\forall i\in\mathcal{V}$. Furthermore,  we define $\mathbf{q}^k=((q_1^k)^T,\ldots,(q_N^k)^T)^T=W^{\frac{1}{2}}\mathbf{v}^k$ as a substitute for $\mathbf{v}^k$, and $\mathbf{q}^k$ can be ensured to identically stay in the range of $W^{\frac{1}{2}}$ by letting $\sum_{i\in\mathcal{V}}q_i^0=0$. To summarize, SoPro takes the following form: Starting from $\mathbf{q}^0$ satisfying $\sum_{i\in\mathcal{V}}q_i^0=0$,
\begin{align}
&\mathbf{x}^{k+1}=\mathbf{x}^k\!-\!(\nabla^2 f(\mathbf{x}^k)\!+\!D)^{-1}(\nabla f(\mathbf{x}^k)\!+\!\beta W\mathbf{x}^k\!+\!\mathbf{q}^k),\label{xup_sopro}\displaybreak[0]\\
&\mathbf{q}^{k+1}=\mathbf{q}^k+\beta W\mathbf{x}^{k+1},\quad\forall k\ge0,\notag
\end{align}
where $\nabla f(\mathbf{x}^{k})=(\nabla f_1(x_1^k)^{T },\ldots,\nabla f_N(x_N^k)^{T })^{T}$ and $\nabla^2 f(\mathbf{x}^{k})=\text{diag}(\nabla^2 f_1(x_1^k),\ldots,\nabla^2 f_N(x_N^k))$ satisfying $\nabla^2 f(\mathbf{x}^k)+D\succ\mathbf{O}_{Nd}$.

The primal update of SoPro \eqref{xup_sopro} requires that each agent uses up all its local samples. However, the agents may only be able to access or process a portion of their local samples at one time, especially in the big data scenario. Motivated by this, we consider approximating the gradient $ \nabla f(\mathbf{x}^{k}) $ and the Hessian $\nabla^2 f(\mathbf{x}^{k})$ in \eqref{xup_sopro} via a stochastic gradient $g(\mathbf{x}^{k})$ and a stochastic Hessian $h (\mathbf{x}^{k})$ given by
	\begin{align}
	  & g(\mathbf{x}^{k}) = ( g_1(x_1^k)^{T }, \ldots,  g_N(x_N^k)^{T })^{T},\notag \displaybreak[0] \\
		& \text{where each }g_{i}(x_{i}^{k} ) = \frac{1}{|\mathcal{G}_i^k|} \sum_{j \in \mathcal{G}_i^k} \nabla l_{i,j}(x_i^k, \xi_{i,j}),\label{eq:stgradient}\displaybreak[0] \\
	  & h(\mathbf{x}^{k}) = \operatorname{diag}(h_1(x_1^k),\ldots,h_N(x_N^k)),\notag \displaybreak[0]\\
		& \text{where each } h_i(x_{i}^{k}) = \frac{1}{|\mathcal{S}_i^k|}\sum_{j \in \mathcal{S}_i^k} \nabla^2 l_{i,j}(x_i^k, \xi_{i,j}).\label{eq:stHessian}
	\end{align}
Here, for each agent $i \in \mathcal{V}$, $\mathcal{G}_i^k$ and $\mathcal{S}_i^k$ are two independent random sample sets uniformly chosen from $\{1,\ldots,\mathcal{C}_i$\} without replacement, so that $g(\mathbf{x}^{k})$ and $h(\mathbf{x}^{k})$ are unbiased, i.e.,
\begin{align*}
\mathbf{E}_{\mathcal{G}_i^k}[g_{i}(\mathbf{x}^{k} )] = \nabla f_{i}(\mathbf{x}^{k} ),\quad\mathbf{E}_{\mathcal{S}_i^k}[h_{i}(\mathbf{x}^{k} )] = \nabla^2 f_{i}(\mathbf{x}^{k})
\end{align*}
for all $\mathbf{x}^{k}$. Due to \eqref{hessian inequality}, $m_i I_d \preceq  h_i(x_i^k) \preceq M_i I_d$ $\forall i\in\mathcal{V}$ and
\vspace{-0.15cm} \begin{equation}
\Lambda_m \preceq  h(\mathbf{x}^k)   \preceq \Lambda_M,\label{bound h}
\vspace{-0.15cm} \end{equation}   
where $\Lambda_{m}=\operatorname{diag}\left(m_{1}, \ldots, m_{N}\right) \otimes I_{d} \succeq \mathbf{O}_{N d}$ and $\Lambda_{M}=\operatorname{diag}\left(M_{1}, \ldots, M_{N}\right) \otimes I_{d} \succ \mathbf{O}_{N d}$.
                                           
Using the above randomly sampled gradient and Hessian, we obtain the following stochastic variant of SoPro: Starting from any $\mathbf{q}^0$ such that $\sum_{i\in\mathcal{V}}q_i^0=0$,
\begin{align}
&\mathbf{x}^{k+1}  =\mathbf{x}^{k}-(h(\mathbf{x}^{k})+D)^{-1}  (g(\mathbf{x}^{k})+\beta W\mathbf{x}^{k}+\mathbf{q}^{k}),
\label{xup_Ssopro1}\displaybreak[0]\\
&\mathbf{q}^{k+1}=\mathbf{q}^k+\beta W\mathbf{x}^{k+1},\quad\forall k\ge0,\notag
\end{align}
where each $D_i$ satisfies $h_i(x)+D_i \succ \mathbf{O}_d$ $\forall x \in \mathbb{R}^d$, i.e., $h(\mathbf{x})+D\succ\mathbf{O}_{Nd}$ $\forall\mathbf{x}\in\mathbb{R}^{Nd}$, so that \eqref{xup_Ssopro1} is well-posed. The above initialization and updates compose our proposed \emph{stochastic second-order proximal} (St-SoPro) method. The distributed implementation of St-SoPro over the undirected network $(\mathcal{V},\mathcal{E})$ is described in Algorithm~\ref{algorithm}, in which $y_i^k$ $\forall i\in\mathcal{V}$ are auxiliary variables for better presentation.

{
		\renewcommand{\baselinestretch}{1.05}
		\begin{algorithm}[ht] 
			\caption{St-SoPro} 
			\begin{algorithmic}[1]\label{algorithm}
			    \STATE \textbf{Initialization:}
			    \STATE Each agent $i \in \mathcal{V}$ sets $q_i^0 \in \mathbb{R}^d$ such that $\sum_{i \in \mathcal{V}} q_i^0 = \mathbf{0}_d$ (or simply sets $q_i^0=\mathbf{0}_d$).
			    \STATE Each agent $i \in \mathcal{V}$ arbitrarily sets $x_i^0 \in \mathbb{R}^d$, and sends $x_i^0$ to each neighbor $j \in \mathcal{N}_i$.
				\STATE After receiving $x _j^0$ $\forall j \in \mathcal{N}_i$, each agent $i \in \mathcal{V}$ sets $y_i^0 = \sum_{j\in \mathcal{N}_i} p_{ij}(x_i^0 - x_j^0)$.
				\FOR{$k = 0,1,2,\ldots $} 
				\STATE Each agent $i\in \mathcal{V}$ randomly and uniformly chooses two independent subsets $\mathcal{G}_i^k $ and $\mathcal{S}_i^k $ of $\{1,\ldots,\mathcal{C}_i\}$, and then computes $g_i(x_i^k)$ and $h_i(x_i^k)$ according to \eqref{eq:stgradient} and \eqref{eq:stHessian}.
			  \STATE Each agent $i\in\mathcal{V}$ computes $x_i^{k+1} = x_i^{k} - (h_i(x_i^k)+D_i)^{-1}(g_i(x_i^k) + \beta y_i^k + q_i^k)$, and then sends $x_i^{k+1}$ to each neighbor $j\in \mathcal{N}_i$.
				\STATE After receiving $x_j^{k+1}$ $\forall j\in \mathcal{N}_i$, each agent $i\in \mathcal{V}$ computes  $y_i^{k+1} = \sum_{j\in \mathcal{N}_i} p_{ij}(x_i^{k+1} - x_j^{k+1})$ and $q_i^{k+1}=q_{i}^{k} + \beta y_i^{k+1}$.
				\ENDFOR
			\end{algorithmic}
		\end{algorithm}
	}
	
	\section{Convergence Analysis}\label{sec: convergence}
	This section provides the convergence analysis of St-SoPro. 
	
	We first impose an assumption on the expected deviation of each sample loss $l_{i,j}(x_i,\xi_{i,j})$ from each entire local loss $f_i(x_i)$ in terms of their gradients.
	\begin{assumption}\label{asm: 2} The random vectors $\xi_{i,j}$ $\forall i\in \mathcal{V}$ $\forall j=1,\ldots,\mathcal{C}_i$ are independent and there is some $\sigma>0$ such that
	\begin{equation*}
	\mathbf{E}_{\xi_{i,j}}\left[\|\nabla l_{i,j}(x_i, \xi_{i,j})-\nabla f_{i}(x_i)\|^{2}\right]\leq \sigma^{2},\quad\forall x_i\in\mathbb{R}^d.
	\end{equation*}
	\end{assumption}	
To simplify the notation, below we let $\mathcal{C}_i=\mathcal{C}$ $\forall i\in\mathcal{V}$ and $\mathcal{G}_i^k$ $\forall i\in\mathcal{V}$ $\forall k\ge0$ be of the same size $\mathcal{G}$. We also abbreviate $\mathbf{E}_{\mathcal{G}_i^k}[\cdot]$ and $\mathbf{E}_{\mathcal{S}_i^k}[\cdot]$ to $\mathbf{E}[\cdot]$. According to \cite[Chapter 2]{LohrSampling}, 
	\begin{align}
	&\mathbf{E} \left[ \|g(\mathbf{x}^{k})-\nabla f(\mathbf{x}^{k})\|^{2}\right]=\sum_{i \in \mathcal{V}} \mathbf{E} \left[ \|g_i(x_i^{k})-\nabla f_i(x_i^{k}) \|^{2}  \right] \notag\displaybreak[0]\\
	&\qquad\qquad\qquad\qquad\qquad\leq N \tau \sigma^2,\text{ where } \tau \coloneqq \frac{\mathcal{C}-\mathcal{G}}{\mathcal{C}\mathcal{G}}.\label{tau sigma^2}
	\end{align}
This is consistent with the fact that when computing the stochastic gradient $g(\mathbf{x}^{k})$, if we reduce the number $\mathcal{G}$ of randomly selected samples, then the discrepancy between $g(\mathbf{x}^{k})$ and $\nabla f(\mathbf{x}^{k})$ becomes larger in expectation.
	
Next, for the sake of presenting the convergence result, we introduce the following notation and definitions. According to \cite{wuxuyang_sopro}, any $\mathbf{v}\in \mathbb{R}^{Nd} $ satisfying $\nabla f\left(\mathbf{x}^{\star}\right)=-W^{\frac{1}{2}} \mathbf{v}$ is a dual optimum of problem~\eqref{rereformulation}. Thus, we define
	\vspace{-0.15cm} \begin{equation}
	\mathbf{v}^{\star}=-\left(W^{\dagger}\right)^{\frac{1}{2}} \nabla f\left(\mathbf{x}^{\star}\right)\label{vstar1}
    \vspace{-0.15cm} \end{equation} 
as a particular dual optimum of \eqref{rereformulation}. Also, throughout this section, we let $\mathbf{v}^{k}=(W^{\dagger})^{\frac{1}{2}} \mathbf{q}^{k}$, $\mathbf{z}^{k}=((\mathbf{x}^{k})^{T},(\mathbf{v}^{k})^{T})^{T}$, and $\mathbf{z}^{\star}=((\mathbf{x}^{\star})^{T},(\mathbf{v}^{\star})^{T})^{T}$. Since such $\mathbf{v}^{k}$ satisfies \eqref{vup1},
    \vspace{-0.15cm} \begin{equation}
    \mathbf{v}^{k}, \mathbf{v}^{\star}, \mathbf{v}^{k}-\mathbf{v}^{\star} \in \{\mathbf{x}\in\mathbb{R}^{Nd}|x_1+\cdots+x_N=\mathbf{0}_d\}\label{v in s^o}.
    \vspace{-0.15cm} \end{equation} 
In addition, we define $f_{\beta}(\mathbf{x})=f(\mathbf{x}) + \frac{\beta}{4} \|\mathbf{x}\|_{W}^2 $. Based on \cite[Lemma 1]{wuxuyang_sopro}, for any $\mathbf{x} \in \mathbb{R}^{Nd}$,
	\vspace{-0.15cm} \begin{equation}
   \left\langle\nabla f_{\beta}(\mathbf{x})-\nabla f_{\beta}\left(\mathbf{x}^{\star}\right), \mathbf{x}-\mathbf{x}^{\star}\right\rangle \geq \zeta(\gamma) \left\|\mathbf{x}-\mathbf{x}^{\star}\right\|^{2},\label{restricted stongly}
    \vspace{-0.15cm} \end{equation}
where $\zeta:(0,\infty)\rightarrow\mathbb{R}$ is given by $\zeta(\gamma)  = \min \{ \frac{m_{\Bar{f}}}{N} - 2M\gamma, \beta \frac{\lambda_W}{2(1+1/\gamma^2)}\}$, $m_{\Bar{f}}$ is given in Assumption~\ref{asm: 1}a), $M=\max_{i\in\mathcal{V}}M_i>0$, and $\lambda_W$ is the smallest non-zero eigenvalue of $W$. It can be shown that $\zeta(\gamma)>0$ if and only if $\gamma\in (0,m_{\Bar{f}}/(2MN))$, and its maximum value is attained at the unique positive root of the cubic equation $4 M N\gamma^3 + (\beta N \lambda_W - 2m_{\Bar{f}})\gamma^2 + 4MN \gamma - 2m_{\Bar{f}} = 0$. We denote the maximum value of $\zeta(\gamma)$ by $m_{\beta}$, which is the convexity parameter of $f_{\beta}$ (and indeed can be taken as any positive $\zeta(\gamma)$). 

Our convergence analysis relies on the following parameter condition. Suppose there exists $\eta_s \in(0,1)$ such that
\vspace{-0.15cm} \begin{equation}
D \succ \frac{\Lambda_{M}}{2(1-\eta_s)}+\frac{\left(\Lambda_{M}-\Lambda_{m}\right)^{2}}{8 \eta_s m_{\beta} } +\frac{\Lambda_{M}-3\Lambda_{m}}{2} + \beta(\frac{{I}_{Nd}}{2}+ W).\label{D}
\vspace{-0.15cm} \end{equation}
Let $R=\frac{\Lambda_m + \Lambda_M}{2}+D$ and $Q=\operatorname{diag}(\beta R,I_{N d})$, which are guaranteed to be positive definite due to \eqref{D}. Furthermore, it follows from \eqref{bound h} and \eqref{D} that the condition $h(\mathbf{x})+D \succ \mathbf{O}_{Nd}$ required in Section~\ref{sec:algorithm} holds. 

We provide our main result in the theorem below.
    
    \begin{theorem} \label{theo: A-1}
	 Suppose Assumptions \ref{asm: 1} and \ref{asm: 2} hold. If \eqref{D} holds for some $\eta_s \in(0,1)$, then $\mathbf{z}^{k}$ converges linearly to a neighborhood of $\mathbf{z}^{\star}$ in expectation, i.e., there exist $\delta_s \in(0,1)$ and $\Gamma>0$ such that for each $k \geq 0$,
	\vspace{-0.15cm} \begin{equation}
		\!\!\!\mathbf{E}\left[\|\mathbf{z}^{k+1}\!-\!\mathbf{z}^{\star}\|_{Q}^{2}\right]\!\leq\!(1\!-\!\delta_s)\mathbf{E}\left[\|\mathbf{z}^{k}\!-\!\mathbf{z}^{\star}\|_{Q}^{2}\right]\!+\!\Gamma N \tau \sigma^2,\label{rate}
	\vspace{-0.15cm} \end{equation}
	\vspace{-0.15cm} \begin{equation}
        \operatorname{lim\;sup}_{k \rightarrow \infty}\mathbf{E}\left[\|\mathbf{z}^{k}-\mathbf{z}^{\star}\|_{Q}^{2} \right]\leq \frac{1}{\delta_s}\Gamma N \tau \sigma^2.\label{neighborhood}
    \vspace{-0.15cm} \end{equation}
    In particular, given any $c_{1}>0$, $\Gamma = \frac{2 \left(1+ c_{1}\right) \delta_s }{\lambda_{W}}+2$ and 
    \begin{align}
	    \delta_s=& \sup_{c_{2}>0} \min \left\{\tfrac{\beta \lambda_{W} \kappa_{c_0, \eta_s}}{2\left(1+c_{1}\right)\left\|\Lambda_{M}+D\right\|^{2}},\tfrac{1-\eta_s}{\left(1+1 / c_{1}\right)\left(1+c_{2}\right)},\right.\nonumber\displaybreak[0]\\
	    &\left.\tfrac{2 \eta_s m_{\beta}-c_0}{\lambda_{\max }\left(R+\left(1+1 / c_{1}\right)\left(1+1 / c_{2}\right) \Lambda_{M}^{2} /\left(\beta \lambda_{W}\right)\right)}\right\},\label{delta_s}
    \end{align}
where $c_0 \in (0, 2\eta_s m_{\beta})$ is such that $\kappa_{c_0,\eta_s}\coloneqq\lambda_{min}(R-\frac{\Lambda_{M}}{2(1-\eta_s)}-\frac{(\Lambda_{M}-\Lambda_{m})^{2}}{4 c_0}+\Lambda_{m}-\Lambda_{M}-\beta(\frac{I_{N d}}{2}+W))>0$ (which always exists).
    \end{theorem}

\begin{proof}
See Appendix~\ref{ssec:proof}.
\end{proof}

It can be shown that larger $\min_{i\in\mathcal{V}}m_i$, smaller $\max_{i\in\mathcal{V}}M_i$, or smaller $\lambda_W$ leads to faster convergence (i.e., larger $\delta_s$) of St-SoPro. Such analysis follows the idea in \cite{wuxuyang_sopro} and is thus omitted here due to space limitation. In addition, it can be seen from \eqref{neighborhood} that the error bound of $\operatorname{lim\;sup}_{k \rightarrow \infty}\mathbf{E}\left[\|\mathbf{z}^{k}-\mathbf{z}^{\star}\|_{Q}^{2} \right]$ drops with the decrease of $\tau$. Hence, essentially, larger random sample sets for computing the stochastic gradients lead to smaller optimality error.

In fact, the expected distance between $\mathbf{x}^k $ and $\mathbf{x}^{\star}$ can eventually reach an arbitrarily small value under proper parameter settings. To see this, for simplicity, let $m_i=m>0$ and $M_i=M\ge m$ for all $i\in\mathcal{V}$, and pick any $\eta_s\in(0,1)$, $c_0\in(0,2\eta_s m_{\beta})$, and $c_1>0$. We choose, for example, $D=\alpha I_{N d}$ with $\alpha=(\frac{1}{2}+\lambda_{max}(W))\beta + \mu$, $\beta>0$, and $\mu>\tfrac{M-3m}{2} + \tfrac{{M}}{2(1-\eta_s)}+\tfrac{({M}-{m})^{2}}{4 c_0}$, so that \eqref{D} holds. This also results in $\kappa_{c_0,\eta_s}=\mu-\frac{M-3m}{2}-\frac{{M}}{2(1-\eta_s)}-\frac{({M}-{m})^{2}}{4 c_0}>0$. From \eqref{neighborhood}, we obtain $\operatorname{lim\;sup}_{k \rightarrow \infty}\mathbf{E}[\|\mathbf{x}^{k}-\mathbf{x}^{\star}\|^{2} ]\leq \frac{\Gamma N \tau \sigma^2}{\delta_s \beta [(m+M)/2 + \beta(1/2 + \lambda_{max}(W)) + \mu]}$. It can thus be shown that as $\beta\rightarrow\infty$, such an upper bound on $\operatorname{lim\;sup}_{k \rightarrow \infty}\mathbf{E}\left[\|\mathbf{x}^{k}-\mathbf{x}^{\star}\|^{2} \right]$ goes to zero. Since this bound is continuous at $\beta$, given any $\epsilon>0$, the above parameter setting with a sufficiently large $\beta$ guarantees $\operatorname{lim\;sup}_{k \rightarrow \infty}\mathbf{E}\left[\|\mathbf{x}^{k}-\mathbf{x}^{\star}\|^{2} \right]<\epsilon$.

\section{Numerical Experiment}
This section compares the practical convergence performance of St-SoPro with several state-of-the-art distributed stochastic optimization algorithms.

In the numerical experiment, we intend to learn linear classifiers by solving $l_2$-regularized logistic regression of the following form over a randomly generated, undirected, and connected network:
  \vspace{-0.15cm} \begin{equation} \label{experiment}
	    \underset{x\in \mathbb{R}^d}{\min} \sum_{i\in \mathcal{V}} \frac{1}{\mathcal{C}_i} \sum_{j=1}^{\mathcal{C}_i} \left( \frac{\lambda}{2}\|x\|^2 +\ln(1+e^{-(a_{i,j}^{\mathrm{T}}x) b_{i,j}})\right), 
	\vspace{-0.15cm} \end{equation}
where $\lambda > 0$ is the regularization parameter and $\{a_{i,j}, b_{i,j}\}$ are the data samples. Our experiment is conducted on two standard real datasets \emph{a4a} and \emph{mushrooms} from the LIBSVM library \cite{chang2011libsvm}.
Table~\ref{tab:parameters} lists the values of the problem and network parameters corresponding to these two datasets, including the problem dimension $d$, the number $N$ of agents, the network's average degree $d_a=\sum_{i\in \mathcal{V}}\frac{|\mathcal{N}_i|}{N}$, the total number $\mathcal{C}_i$ of samples that we assign to each agent $i$, the sizes $|\mathcal{G}_i^k|$ and $|\mathcal{S}_i^k|$ of the random sample sets that each agent $i$ chooses per iteration, as well as the regularization parameter $\lambda$.	 

\begin{table}[ht]
    \centering
    \caption{Parameter values in the numerical experiment.}
    \begin{tabular}{c|c c c c c c c }
    \hline
     & d & $N$ & $d_a$& $\mathcal{C}_i$ & $|\mathcal{G}_i^k|$& $|\mathcal{S}_i^k|$ & $ \lambda$\\\hline
     \emph{a4a} & 123 & 20 & 5 & 239 & 80 & 10 & $ 10^{-2} $\\
     \emph{mushrooms} & 112 & 10 & 3 & 600 & 80 & 25 & $ 10^{-2} $\\
     \hline     
    \end{tabular}
    \label{tab:parameters}
\end{table}

The simulations include DSGD \cite{pu2021sharp}, EDAS \cite{huang2022improving}, DSGT \cite{pu2021distributed}, and DPD-SGD-T \cite{yi2022primal}, which are all first-order methods, for comparison with our proposed St-SoPro. We fine-tune all the algorithm parameters so that the algorithms reach a given accuracy ($2\times 10^{-1}$ for \emph{a4a} and $10^{-1}$ for \emph{mushrooms}) within fewest possible iterations.

Figures~\ref{fig:a4a}(a)--(c) and~\ref{fig:mushrooms}(a)--(c) plot the evolutions of the optimality error $\frac{1}{N}\sum_{i \in \mathcal{V}}\|{x}_i^k - {x}^{\star}\|^2$ generated by the aforementioned algorithms over \emph{a4a} and \emph{mushrooms} with respect to the number of iterations, the number of communication bits (set as $32$ times the number of transmitted real scalars according to \cite{Alistarh2017QSGDCS}), and computation time. Observe that St-SoPro converges faster than the other algorithms to reach the given accuracy, validating its computational and communication efficiency. It is worth mentioning that although St-SoPro is a second-order method, its computational cost can be comparable with the first-order methods when addressing such common machine learning problems. Figures~\ref{fig:a4a}(d) and~\ref{fig:mushrooms}(d) present the correct rates of classification on the test sets upon completing each iteration of these algorithms as training, whereby St-SoPro outperforms the others in the training effect.

\begin{figure*}[ht]
	\setlength{\abovecaptionskip}{-5pt}
	\setlength{\belowcaptionskip}{-10pt}
	\centering
	\begin{minipage}{0.24\linewidth}
	\begin{subfigure}[]{
	    \centering
		\includegraphics[width=1.67in,height=1.15in]{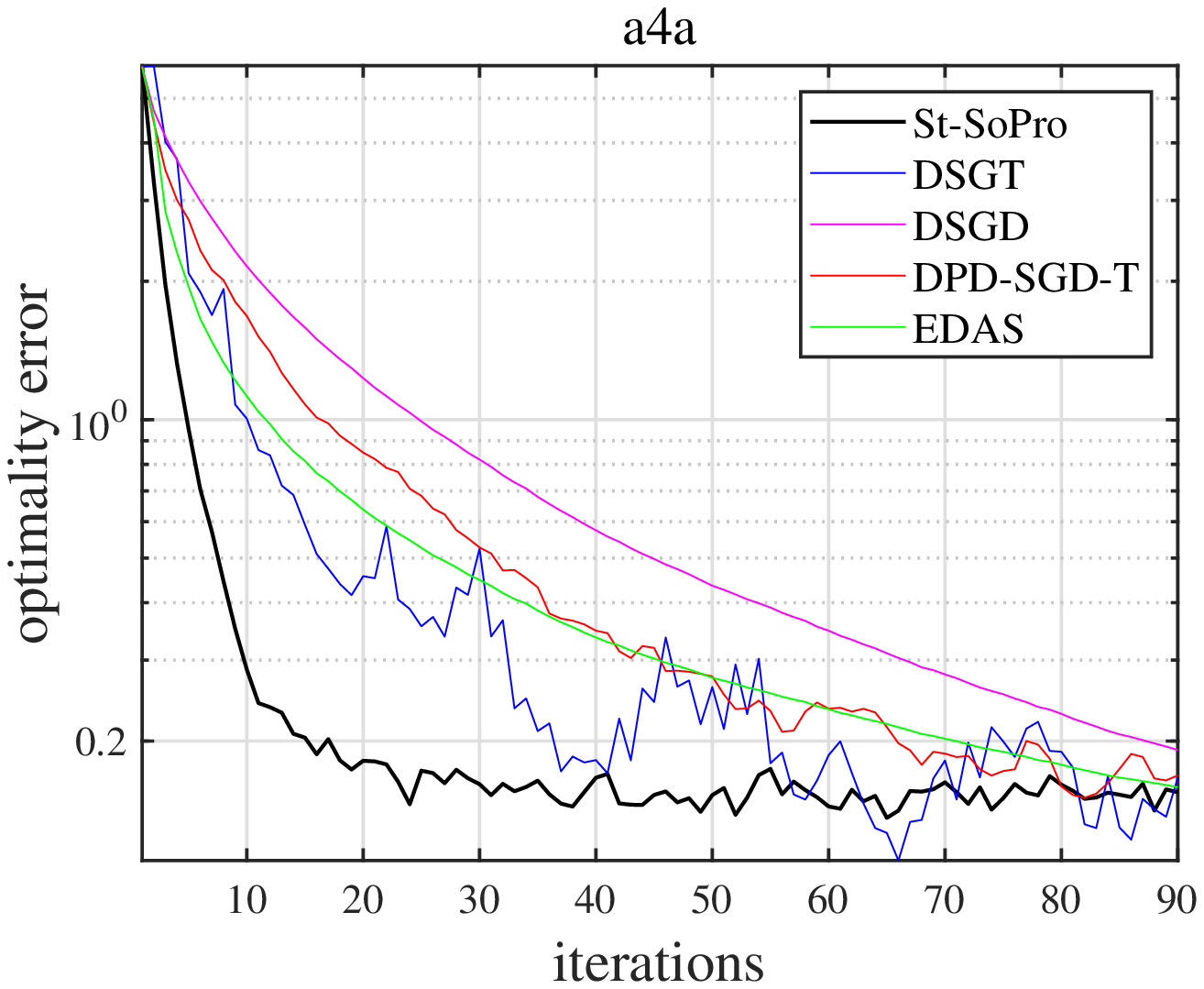}
		\label{a4a_iter}}
	\end{subfigure}
	\end{minipage}%
	\begin{minipage}{0.24\linewidth}
	\begin{subfigure}[]{
		\centering
		\includegraphics[width=1.67in,height=1.15in]{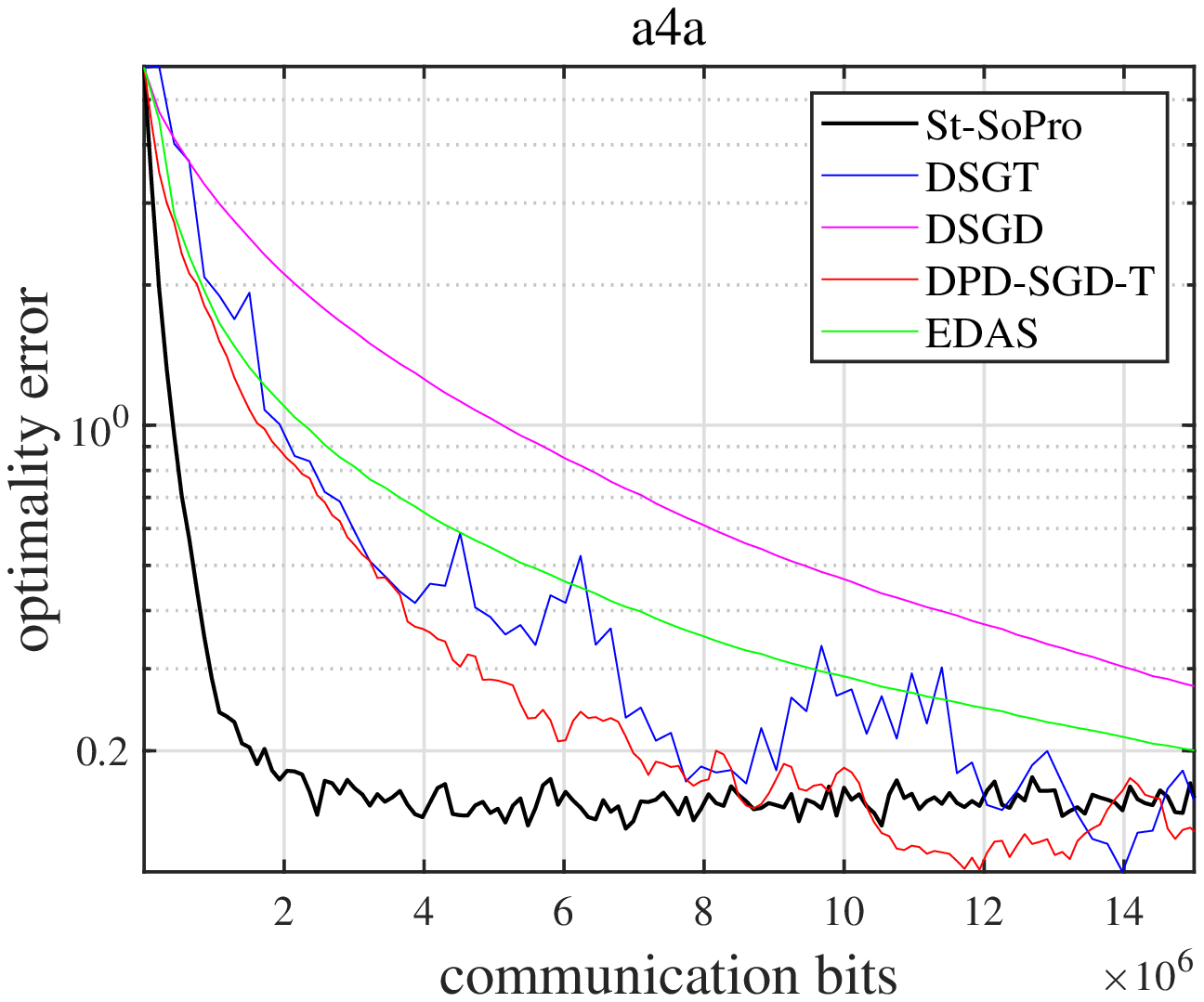}
		\label{a4a_comm}
		}\end{subfigure}
	\end{minipage}
	\begin{minipage}{0.24\linewidth}
	\begin{subfigure}[]{
		\centering
		\includegraphics[width=1.67in,height=1.15in]{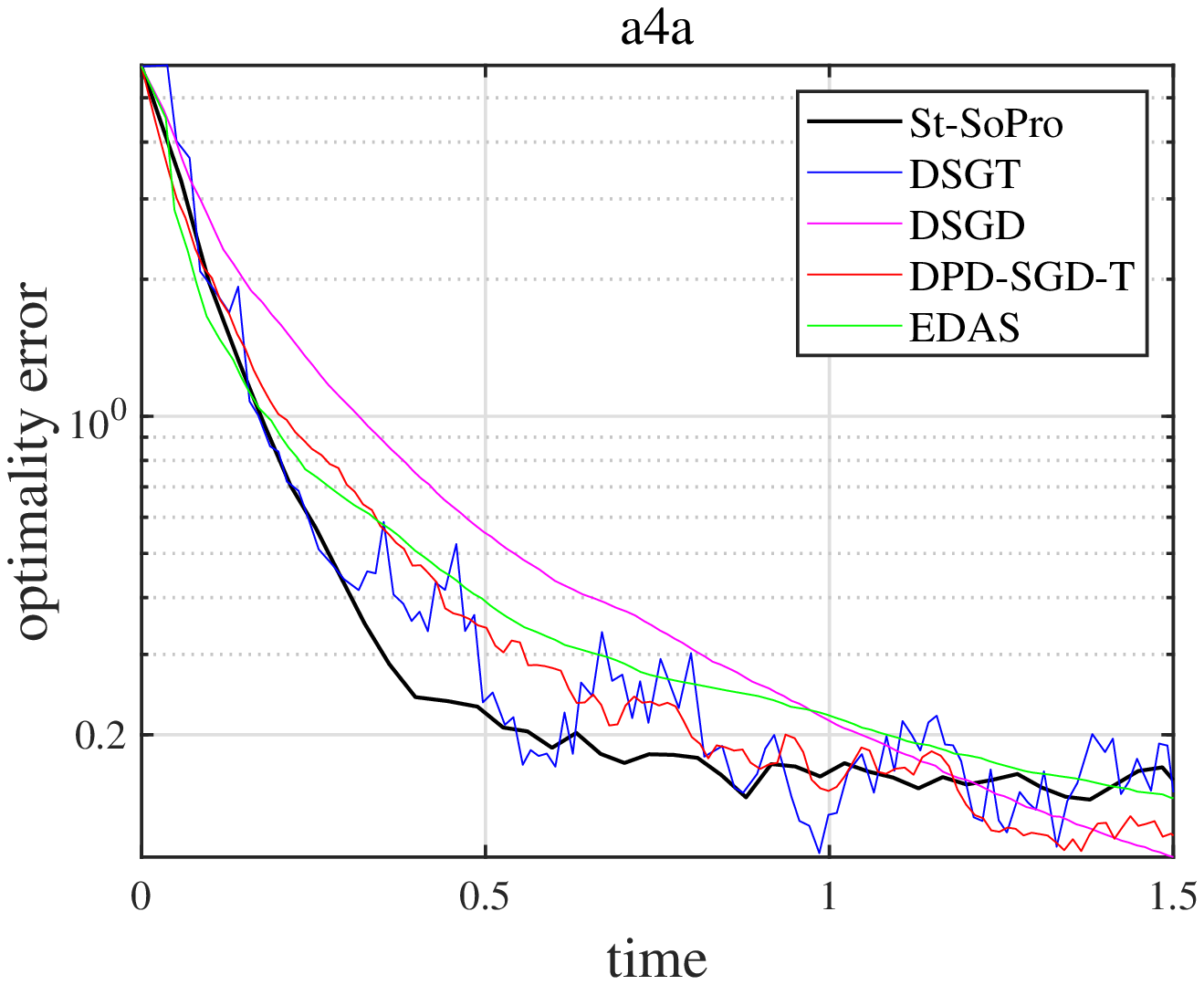}
		\label{a4a_time}
		}\end{subfigure}
	\end{minipage}
	\begin{minipage}{0.24\linewidth}
	\begin{subfigure}[]{
		\centering
		\includegraphics[width=1.67in,height=1.15in]{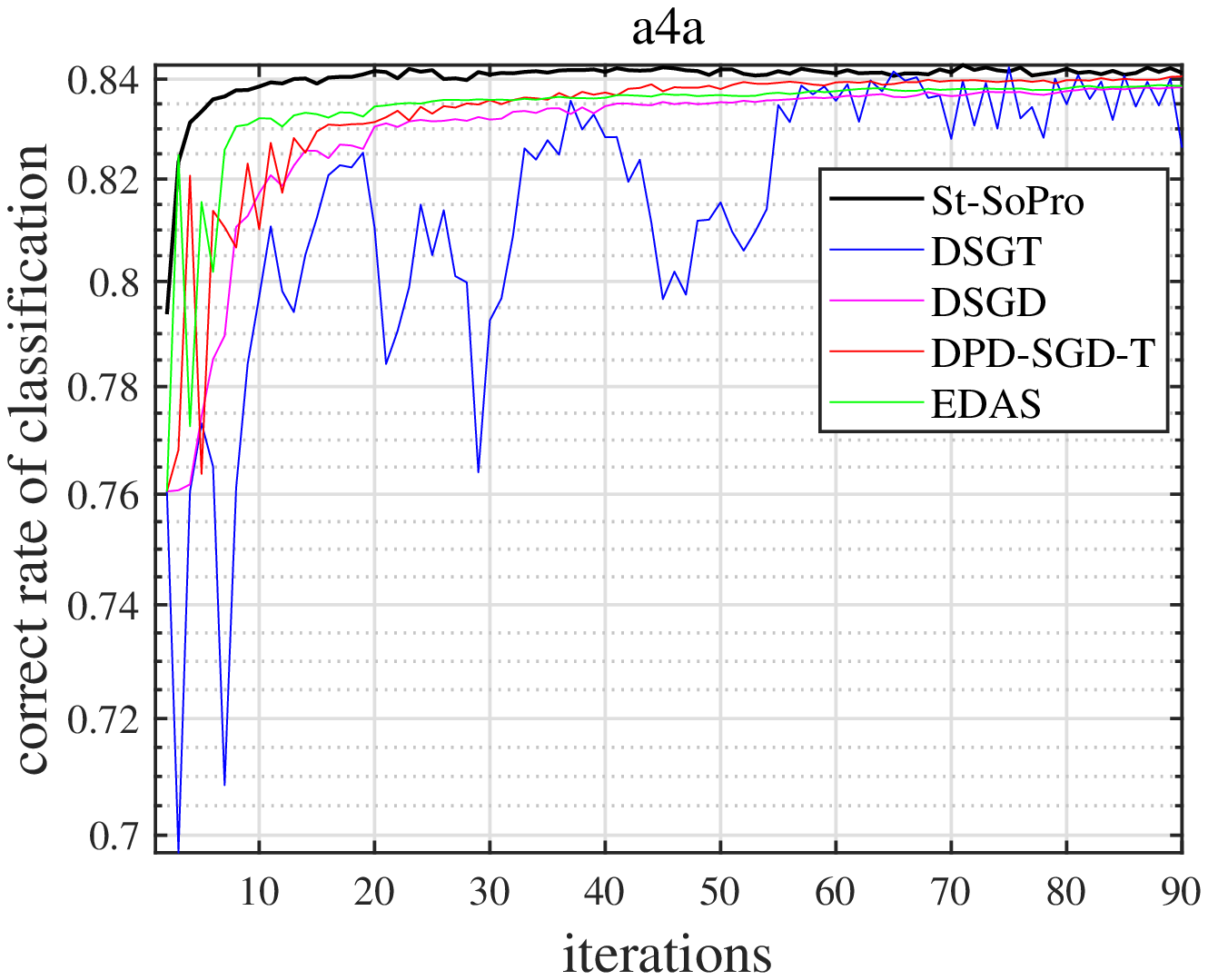}
		\label{a4a_acc}
		}\end{subfigure}
	\end{minipage}
	\caption{Convergence performance of St-SoPro, DSGT, DSGD, DPD-SGD-T, and EDAS on dataset \emph{a4a}.\label{fig:a4a}}
\end{figure*}

\begin{figure*}[ht]
	\setlength{\abovecaptionskip}{-5pt}
	\setlength{\belowcaptionskip}{-10pt}
	\centering
	\begin{minipage}{0.24\linewidth}
	\begin{subfigure}[]{
	    \centering
		\includegraphics[width=1.67in,height=1.15in]{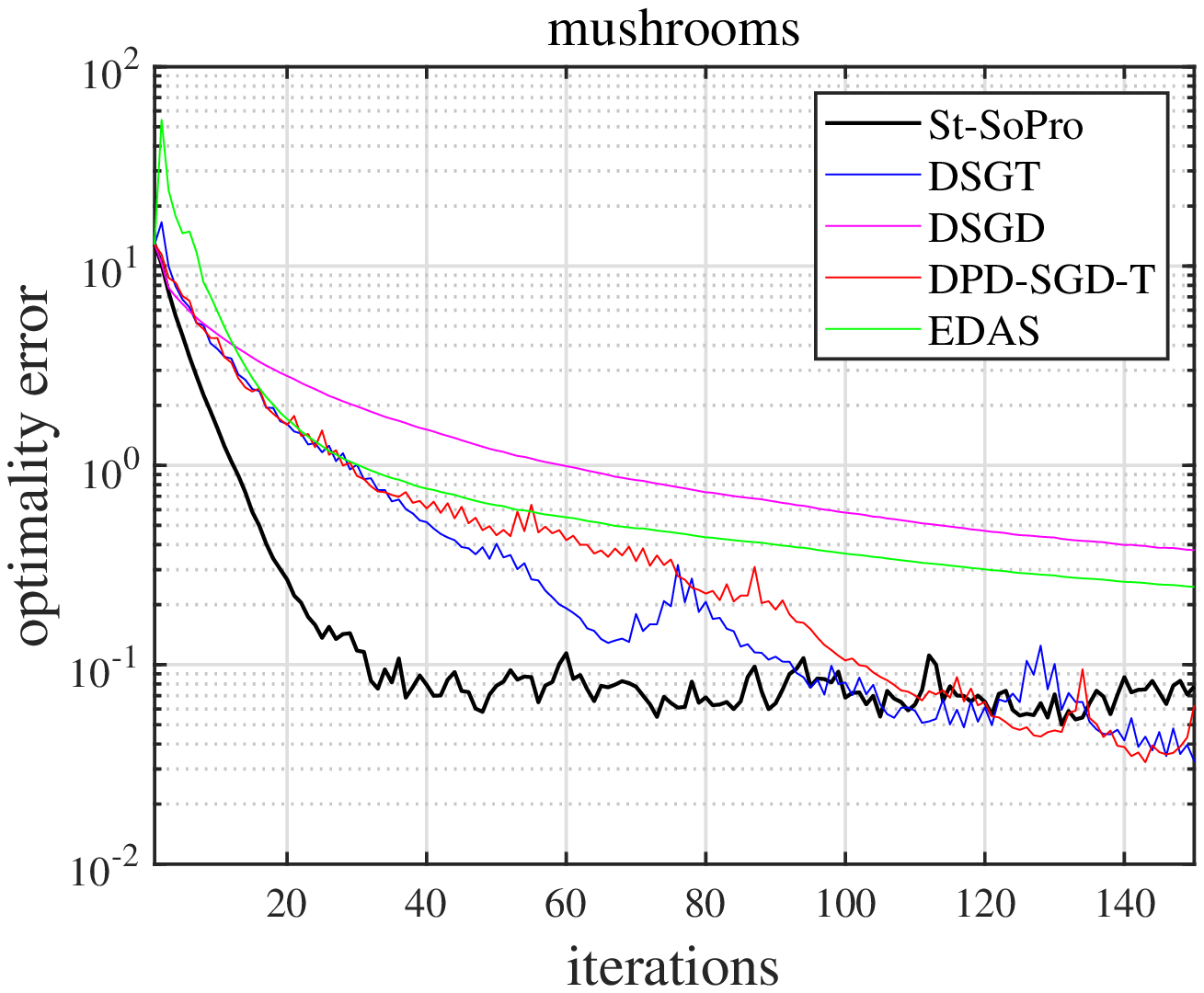}
		\label{mushrooms_iter}}
	\end{subfigure}
	\end{minipage}%
	\begin{minipage}{0.24\linewidth}
	\begin{subfigure}[]{
		\centering
		\includegraphics[width=1.67in,height=1.15in]{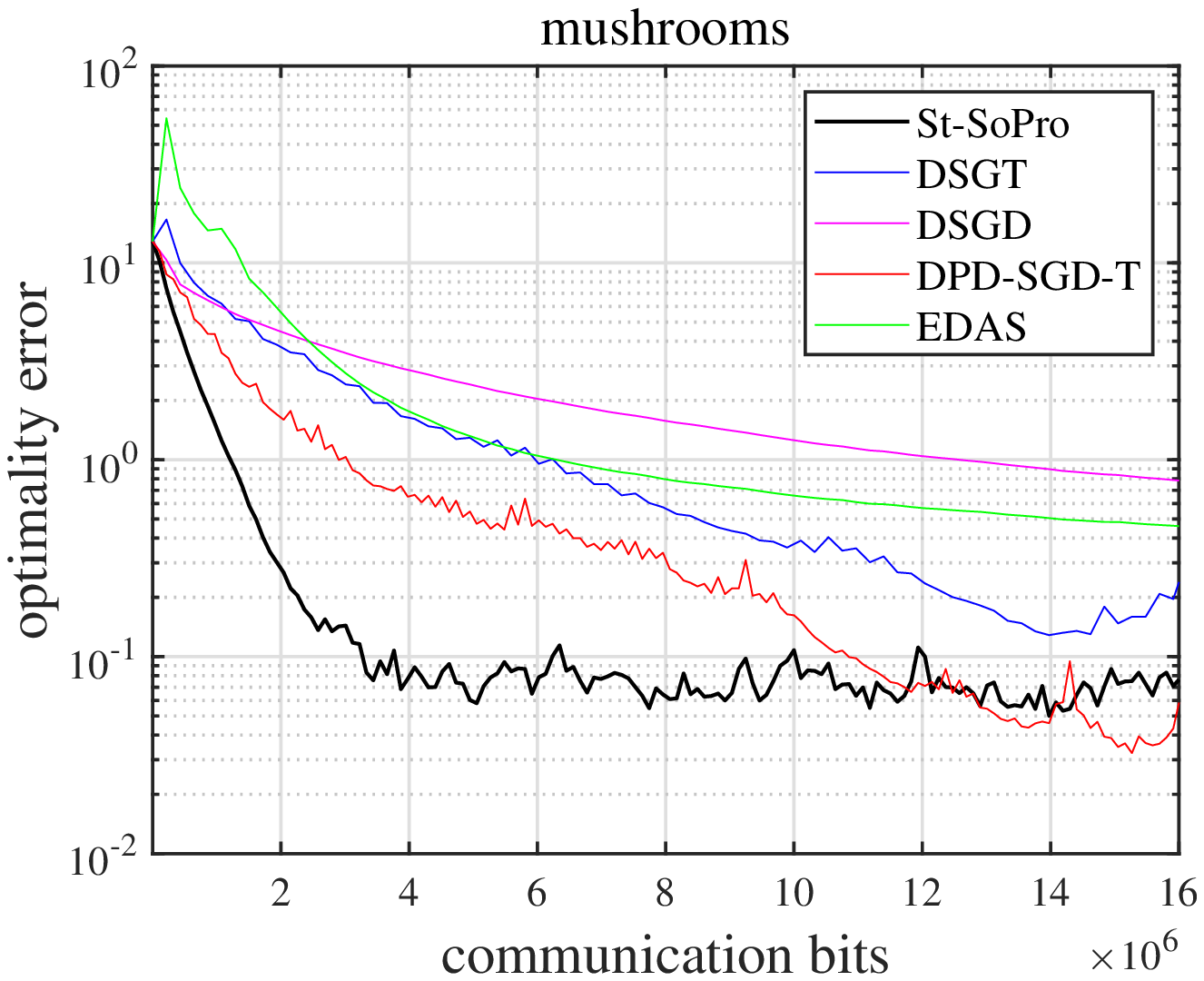}
		\label{mushrooms_comm}
		}\end{subfigure}
	\end{minipage}
	\begin{minipage}{0.24\linewidth}
	\begin{subfigure}[]{
		\centering
		\includegraphics[width=1.67in,height=1.15in]{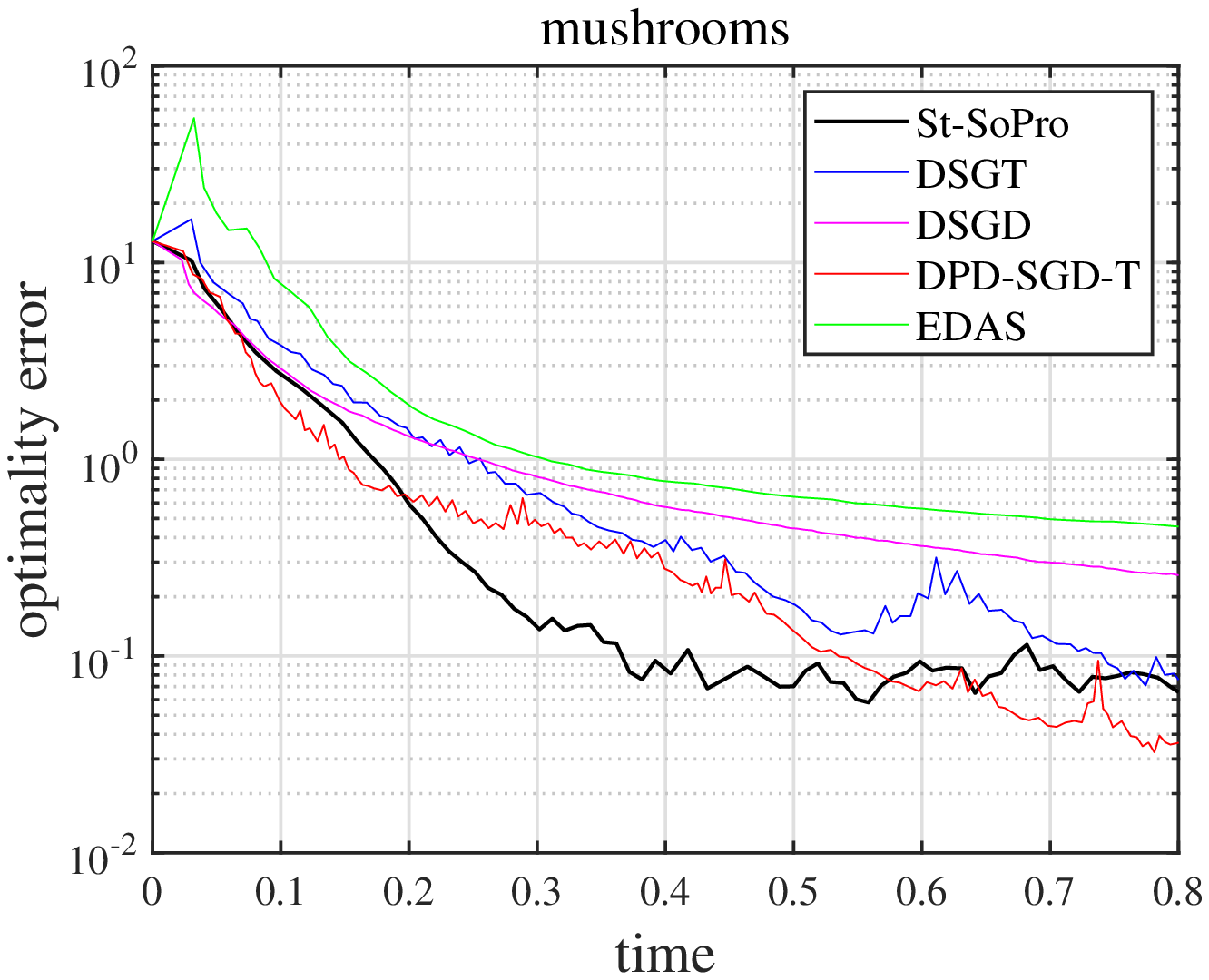}
		\label{mushrooms_time}
		}\end{subfigure}
	\end{minipage}
	\begin{minipage}{0.24\linewidth}
	\begin{subfigure}[]{
		\centering
		\includegraphics[width=1.67in,height=1.15in]{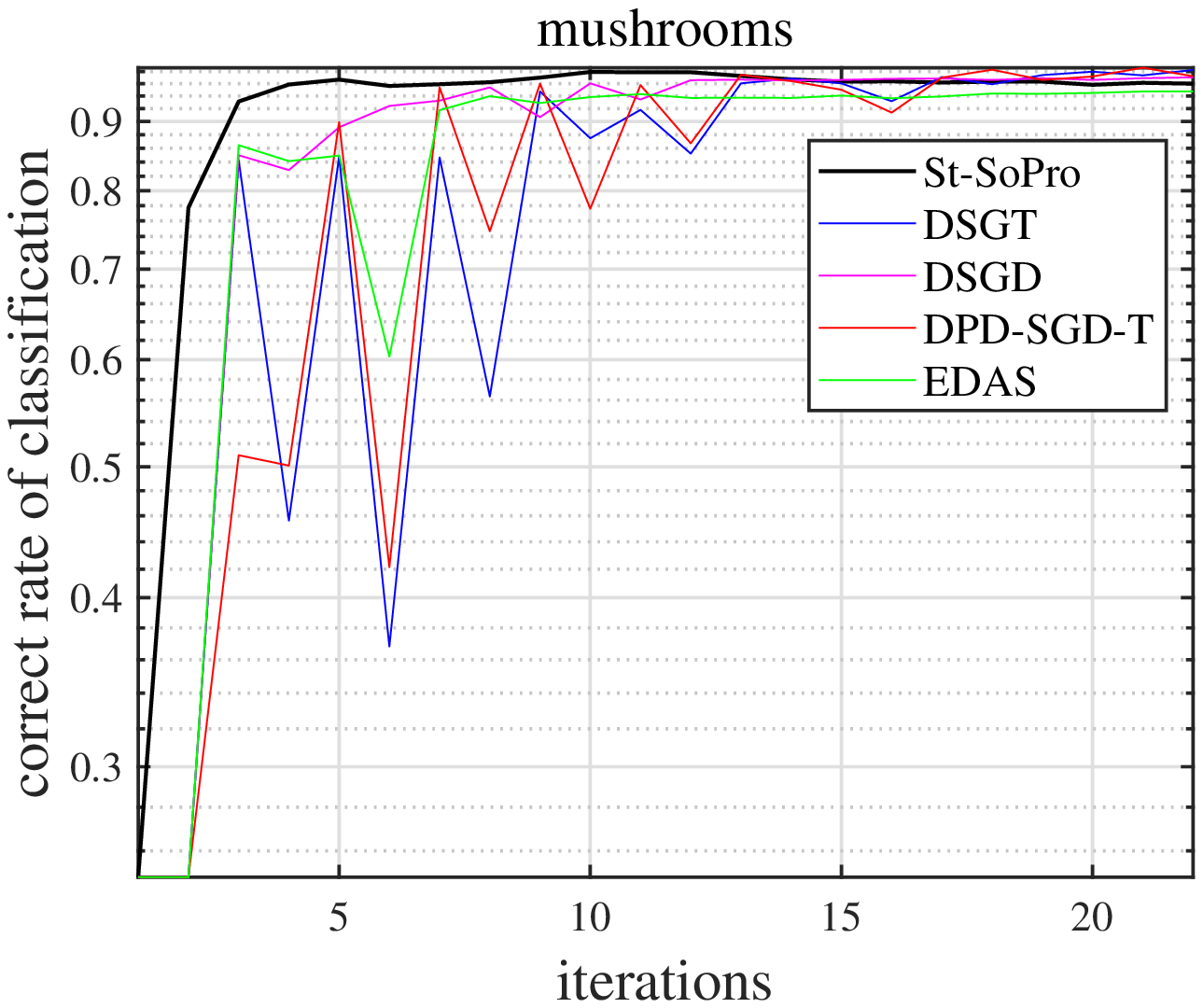}
		\label{mushrooms_acc}
		}\end{subfigure}
	\end{minipage}
	\caption{Convergence performance of St-SoPro, DSGT, DSGD, DPD-SGD-T and EDAS on dataset \emph{mushrooms}.\label{fig:mushrooms}}
\end{figure*}

\section{Conclusion}
We have developed St-SoPro, a distributed stochastic second-order proximal method, for addressing strongly convex and smooth optimization over undirected networks. Different from the existing first-order distributed stochastic algorithms, St-SoPro incorporates a second-order approximation of an augmented Lagrangian function and randomly samples each local gradient and Hessian. We show that St-SoPro linearly converges to a neighborhood of the optimal solution in expectation, and the neighborhood can be arbitrarily small. Simulations over two real datasets demonstrate that St-SoPro is both computationally and communication-wise efficient.

\appendix
\subsection{Proof of Theorem~\ref{theo: A-1}}\label{ssec:proof}

The following lemma intends to bound the difference between $\mathbf{E}\left[\left\|\mathbf{z}^{k}-\mathbf{z}^{\star}\right\|_{Q}^{2}\right]$ and $\mathbf{E}\left[\left\|\mathbf{z}^{k+1}-\mathbf{z}^{\star}\right\|_{Q}^{2}\right]$.

    \begin{lemma}\label{lem: A-1} For each $k\ge0$,
	\begin{align}
		    &\mathbf{E}\left[\left\|\mathbf{z}^{k}-\mathbf{z}^{\star}\right\|_{Q}^{2}\right]-\mathbf{E}\left[\left\|\mathbf{z}^{k+1}-\mathbf{z}^{\star}\right\|_{Q}^{2}\right] \notag \\
		    \geq& \beta\left(2 \eta_s m_{\beta}-c_0 \right) \mathbf{E}\left[\left\|\mathbf{x}^{k}-\mathbf{x}^{\star}\right\|^{2} \right] + \beta^{2}(1-\eta_s) \mathbf{E}\left[ \left\|\mathbf{x}^{k}\right\|_{W}^{2} \right] \notag \\
		    &-\beta \mathbf{E}\left[\left\|\mathbf{x}^{k+1}-\mathbf{x}^{k}\right\|_{\mathcal{A}_{c_0, \eta_s}+\beta W-R}^{2} \right] - 2  N\tau \sigma^2, \label{penultimate}
	\end{align}
    where $c_0>0$ can be arbitrary and $\mathcal{A}_{c_0, \eta_s}\coloneqq\Lambda_{M} /(2(1-\eta_s))+\left(\Lambda_{M}-\Lambda_{m}\right)^{2} /(4 c_0)+\Lambda_{M}-\Lambda_{m} + \frac{\beta I_{N d}}{2}$.
    \end{lemma}
    
\begin{proof}
We first equivalently expand the left-hand side of \eqref{penultimate}. Similar to \cite[Eq. (27)]{wuxuyang_sopro}, we derive 
	\begin{align}
		&\mathbf{E}\left[\left\|\mathbf{z}^{k}-\mathbf{z}^{\star}\right\|_{Q}^{2}\right]-\mathbf{E}\left[\left\|\mathbf{z}^{k+1}-\mathbf{z}^{\star}\right\|_{Q}^{2}\right] \notag\displaybreak[0]\\
		= &\mathbf{E}\left[\left\|\mathbf{z}^{k}-\mathbf{z}^{k+1}\right\|_{Q}^{2} \right] +2\mathbf{E}\left[ \beta\left\langle\mathbf{x}^{k+1}-\mathbf{x}^{\star}, R\left(\mathbf{x}^{k}-\mathbf{x}^{k+1}\right)\right\rangle\right] \notag\displaybreak[0]\\
		& + 2 \mathbf{E}\left[\left\langle\mathbf{v}^{k}-\mathbf{v}^{k+1}, \mathbf{v}^{k+1}-\mathbf{v}^{\star}\right\rangle \right].\label{equivalent expansion}
	\end{align} 
Then, using \eqref{vup1} and $W^{\frac{1}{2}} \mathbf{x}^{\star}=\mathbf{0}_{N d}$, we obtain $\langle\mathbf{v}^{k}-\mathbf{v}^{k+1}, \mathbf{v}^{k+1}-\mathbf{v}^{\star}\rangle=-\beta\langle \mathbf{x}^{k+1} - \mathbf{x}^{\star}, W^{\frac{1}{2}} (\mathbf{v}^{k+1} - \mathbf{v}^{\star})\rangle$. From \eqref{xup_Ssopro1} and \eqref{vup1}, we have $W^{\frac{1}{2}} \mathbf{v}^{k+1}=W^{\frac{1}{2}}\left(\mathbf{v}^{k+1}-\mathbf{v}^{k}\right)+W^{\frac{1}{2}} \mathbf{v}^{k}=(\beta W-\tilde{H}^{k})(\mathbf{x}^{k+1}-\mathbf{x}^{k})-g(\mathbf{x}^{k})$, where $\tilde H^k\coloneqq h(\mathbf{x}^k)+D$. The above two equations, together with \eqref{vstar1}, give
	\begin{align}
		&\langle\mathbf{v}^{k}-\mathbf{v}^{k+1}, \mathbf{v}^{k+1}-\mathbf{v}^{\star}\rangle=\beta\langle\mathbf{x}^{k+1}-\mathbf{x}^{\star}, g(\mathbf{x}^{k})-\nabla f(\mathbf{x}^{\star})\rangle \notag\displaybreak[0]\\
		&\qquad+\beta\langle\mathbf{x}^{k+1}-\mathbf{x}^{\star},(\tilde{H}^{k}-\beta W)(\mathbf{x}^{k+1}-\mathbf{x}^{k})\rangle. \label{v x1}
	\end{align} 
Moreover, based on \eqref{vup1}, $W \mathbf{x}^{\star}=\mathbf{0}_{N d}$, and \cite[Eq. (26)]{wuxuyang_sopro}, 
\begin{align}
&- 2 \beta \langle\mathbf{x}^{k+1}-\mathbf{x}^{\star}, \beta W(\mathbf{x}^{k+1}-\mathbf{x}^k)\rangle  \notag\displaybreak[0]\\
=& - \beta \|\mathbf{x}^{k+1}\|_{\beta W}^2 + \beta \|\mathbf{x}^k\|_{\beta W}^2- \beta \|\mathbf{x}^{k+1}-\mathbf{x}^k\|_{\beta W}^2 \notag\displaybreak[0]\\
=&- \|\mathbf{v}^{k+1}-\mathbf{v}^{k} \|^2 + \beta^2 \|\mathbf{x}^k\|_{ W}^2- \beta\|\mathbf{x}^{k+1}-\mathbf{x}^k\|_{\beta W}^2.\label{26 of 16}
\end{align}
By incorporating \eqref{26 of 16} into \eqref{v x1} and then combining the resulting equation with \eqref{equivalent expansion}, we have
\begin{align}
&\mathbf{E}\left[\|\mathbf{z}^{k}-\mathbf{z}^{\star}\|_{Q}^{2}\right] -\mathbf{E}\left[\|\mathbf{z}^{k+1}-\mathbf{z}^{\star}\|_{Q}^{2}\right]\notag\displaybreak[0]\\
=& 2 \beta \mathbf{E}\left[\langle\mathbf{x}^{k+1}-\mathbf{x}^{\star}, g(\mathbf{x}^{k})-\nabla f(\mathbf{x}^{\star})\rangle\right ]+\beta^{2}\mathbf{E}\left[\|\mathbf{x}^{k}\|_{W}^{2}\right] \notag\displaybreak[0]\\
&+2 \beta\mathbf{E}\left[\langle\mathbf{x}^{k+1}-\mathbf{x}^{\star},(\tilde{H}^{k}-R)(\mathbf{x}^{k+1}-\mathbf{x}^{k})\rangle\right]\notag\displaybreak[0]\\
&+\beta\mathbf{E}\left[\|\mathbf{x}^{k}-\mathbf{x}^{k+1}\|_{R-\beta W}^{2}\right].
\label{pre-lemma}
\end{align}

Subsequently, we provide a lower bound for the first term on the right-hand side of \eqref{pre-lemma}. To do so, we utilize the AM-GM inequality and \eqref{tau sigma^2} to derive
	\begin{align}
		&\mathbf{E} \left[\langle\mathbf{x}^{k+1}-\mathbf{x}^{k}, g(\mathbf{x}^{k})-\nabla f(\mathbf{x}^{\star})\rangle\right] \notag\displaybreak[0]\\ 
		= & \mathbf{E} \left[\langle\mathbf{x}^{k+1}-\mathbf{x}^{k}, \nabla f (\mathbf{x}^{k})-\nabla f(\mathbf{x}^{\star})\rangle\right] \notag\displaybreak[0]\\
		&+ \mathbf{E} \left[\langle\mathbf{x}^{k+1}-\mathbf{x}^{k}, g(\mathbf{x}^{k})-\nabla f(\mathbf{x}^{k})\rangle\right] \notag\displaybreak[0]\\ 
		\geq & -(1-\eta_s)\mathbf{E}\left[\left\|\nabla f(\mathbf{x}^{k}) - \nabla f(\mathbf{x}^{\star})\right\|_{\Lambda_{M}^{-1}}^{2} \right] \notag\displaybreak[0] \\
		&-  \mathbf{E}\left[\left\|\mathbf{x}^{k+1}-\mathbf{x}^{k}\right\|_{\frac{\Lambda_{M}}{4(1-\eta_s)}}^{2}\right] \notag\displaybreak[0]\\
		& -\frac{\beta}{4}\mathbf{E}\left[\left\|\mathbf{x}^{k+1}-\mathbf{x}^{k}\right\|_{}^2\right] - \frac{1}{\beta}\mathbf{E} \left[\left\|g(\mathbf{x}^{k}) - \nabla f(\mathbf{x}^{k})\right\|^2\right] \notag   \notag \displaybreak[0]\\
		\geq & -(1-\eta_s)\mathbf{E}\left[\left\|\nabla f(\mathbf{x}^{k}) - \nabla f(\mathbf{x}^{\star})\right\|_{\Lambda_{M}^{-1}}^{2} \right] \notag\displaybreak[0] \\
		&-  \mathbf{E}\left[\left\|\mathbf{x}^{k+1}-\mathbf{x}^{k}\right\|_{\frac{\Lambda_{M}}{4(1-\eta_s)}+ \frac{\beta I_{Nd}}{4}}^{2}\right] - \frac{1}{\beta}N \tau \sigma^2. \label{the second}
	\end{align}
Due to the Lipschitz continuity of each $\nabla f_{i}$ and the unbiasedness of $g(\mathbf{x}^{k})$, we have $\mathbf{E}[\langle\mathbf{x}^{k}-\mathbf{x}^{\star}, g(\mathbf{x}^{k})-\nabla f(\mathbf{x}^{\star})\rangle]=\mathbf{E}[\langle\mathbf{x}^{k}-\mathbf{x}^{\star}, \nabla f(\mathbf{x}^{k})-\nabla f(\mathbf{x}^{\star})\rangle]\geq\mathbf{E}[\|\nabla f(\mathbf{x}^{k})-\nabla f(\mathbf{x}^{\star})\|_{\Lambda_{M}^{-1}}^{2}]$. We multiply this inequality by $(1-\eta_s)$ and then add it to \eqref{the second}, which leads to
	\begin{align}
		&\mathbf{E} \left[\langle\mathbf{x}^{k+1}-\mathbf{x}^{\star}, g(\mathbf{x}^{k})-\nabla f(\mathbf{x}^{\star})\rangle \right] \notag\displaybreak[0]\\ 
		& - \eta_s \mathbf{E} \left[\langle\mathbf{x}^{k}-\mathbf{x}^{\star}, \nabla f(\mathbf{x}^{k})-\nabla f(\mathbf{x}^{\star})\rangle \right] \notag\displaybreak[0]\\
		\geq &-  \mathbf{E}\left[\left\|\mathbf{x}^{k+1}-\mathbf{x}^{k}\right\|_{\frac{\Lambda_{M}  }{4(1-\eta_s)}+ \frac{\mathbf{\beta I}_{Nd}}{4}}^{2}\right] -\frac{1}{\beta} N \tau \sigma^2. \label{first inequality 1}
	\end{align}
Because of the restricted strong convexity of $f_{\beta}(\mathbf{x}) = f(\mathbf{x}) + \frac{\beta}{4} \|\mathbf{x}\|_{W}^2 $ shown in Section~\ref{sec: convergence} and $W \mathbf{x}^{\star}=\mathbf{0}_{N d}$, we have $\mathbf{E}[\langle\mathbf{x}^{k}-\mathbf{x}^{\star}, \nabla f(\mathbf{x}^{k})-\nabla f(\mathbf{x}^{\star})\rangle]\geq m_{\beta} \mathbf{E}\left[\|\mathbf{x}^{k}-\mathbf{x}^{\star}\|^{2}\right] - \frac{\beta}{2}\mathbf{E}\left[\|\mathbf{x}^{k}\|_{W}^{2}\right]$. This, along with \eqref{first inequality 1}, results in
	\begin{align}
		&\mathbf{E} \left[\langle \mathbf{x}^{k+1}-\mathbf{x}^{\star} , g(\mathbf{x}^{k})-\nabla f(\mathbf{x}^{\star})\rangle \right] \notag\displaybreak[0]\\ 
		\geq &-  \mathbf{E}\left[\left\|\mathbf{x}^{k+1}-\mathbf{x}^{k}\right\|_{\frac{\Lambda_{M}  }{4(1-\eta_s)}+ \frac{\mathbf{\beta I}_{Nd}}{4}}^{2}\right] - \frac{1}{\beta} N \tau \sigma^2 \notag\displaybreak[0]\\
		& + \eta_s m_{\beta}\mathbf{E} \left[\|\mathbf{x}^{k}-\mathbf{x}^{\star}\|^{2} \right] - \frac{\eta_s\beta}{2}\mathbf{E}\left[\|\mathbf{x}^{k}\|_{W}^{2}\right]. \label{lower bound of first term}
	\end{align} 

Next, we bound the third term on the right-hand side of \eqref{pre-lemma}. Because $\tilde H^{k} - R=h(\mathbf{x}^{k})-\frac{\Lambda_{m}+\Lambda_{M}}{2}$ and because of \eqref{bound h}, we have $\frac{\Lambda_{m}-\Lambda_{M}}{2} \preceq \tilde H^{k}-R \preceq \frac{\Lambda_{M}-\Lambda_{m}}{2}$. Let $c_0>0$. Then, similar to \cite[Eq. (30)]{wuxuyang_sopro}, we obtain
	\begin{align}
		&\mathbf{E}\left[\langle\mathbf{x}^{k+1}-\mathbf{x}^{\star},(\tilde H^{k}-R)(\mathbf{x}^{k+1}-\mathbf{x}^{k})\rangle \right] \notag\displaybreak[0]\\
		\geq &-\frac{c_0}{2} \mathbf{E}\left[ \left\|\mathbf{x}^{k}-\mathbf{x}^{\star}\right\|^{2} \right] -\frac{1}{8 c_0} \mathbf{E}\left[ \left\|\mathbf{x}^{k+1}-\mathbf{x}^{k}\right\|_{\left(\Lambda_{M}-\Lambda_{m}\right)^{2}}^{2} \right]  \notag\displaybreak[0]\\
		&-\frac{1}{2} \mathbf{E}\left[ \left\|\mathbf{x}^{k+1}-\mathbf{x}^{k}\right\|_{\Lambda_{M}-\Lambda_{m}}^{2}\right].\label{A-second inequality}
	\end{align}

Combining \eqref{lower bound of first term} and \eqref{A-second inequality} with \eqref{pre-lemma} gives \eqref{penultimate}.
\end{proof}

In addition to Lemma~\ref{lem: A-1}, below we provide an upper bound on $\mathbf{E} \left[\| \mathbf{z}^k - \mathbf{z}^{\star}\|_{Q}^2\right]$. For any $c_1,c_2>0$, through \eqref{tau sigma^2}, \eqref{xup_Ssopro1}, \eqref{vstar1}, \eqref{v in s^o}, and the AM-GM inequality,
\begin{align}
& \mathbf{E} \left [\|\mathbf{v}^{k}-\mathbf{v}^{\star} \|^{2} \right]= \mathbf{E} \left [\| (W^{\dagger} )^{\frac{1}{2}} W^{\frac{1}{2}} (\mathbf{v}^{k}-\mathbf{v}^{\star} ) \|^{2} \right] \notag\displaybreak[0]\\
=& \mathbf{E} \left [\| (W^{\dagger} )^{\frac{1}{2}} \left(\tilde H^{k} (\mathbf{x}^{k}-\mathbf{x}^{k+1} )- \beta W \mathbf{x}^{k}-g (\mathbf{x}^{k} ) +\nabla f (\mathbf{x}^{k} ) \right.\right. \notag\displaybreak[0]\\
& \left.\left. -\nabla f (\mathbf{x}^{k} )+\nabla f (\mathbf{x}^{\star} ) \right) \|^{2}\right] \notag\displaybreak[0]\\
\leq & (1\!+\!c_{1})\mathbf{E} \left [\| (W^{\dagger} )^{\frac{1}{2}}\left( \tilde H^{k} (\mathbf{x}^{k}\!-\!\mathbf{x}^{k+1} )\!-\!g (\mathbf{x}^{k} )\!+\!\nabla f (\mathbf{x}^{k})\right)\|^{2}\right] \notag\displaybreak[0]\\
&+ (1+\frac{1}{c_{1}} ) \mathbf{E} \left [\|( W^{\dagger} )^{\frac{1}{2}}  \left(\beta W \mathbf{x}^{k} +  \nabla f (\mathbf{x}^{k} )-\nabla f (\mathbf{x}^{\star}) \right)  \|^{2} \right] \notag\displaybreak[0]\\
\leq & \frac{2(1+c_{1})}{\lambda_{W}}\mathbf{E} \left [ \|\mathbf{x}^{k+1}-\mathbf{x}^{k} \|_{ (\tilde H^{k} )^{2}}^{2}\right] + \frac{2(1+{c_{1}})}{\lambda_{W}}N \tau \sigma^2 \notag\displaybreak[0]\\
& + \beta^{2} (1+\frac{1}{c_{1}} ) (1+c_{2} ) \mathbf{E} \left [\|\mathbf{x}^{k} \|_{W}^{2} \right] \notag\displaybreak[0]\\
&+\frac{ (1+1 / c_{1} ) (1+1 / c_{2} )}{\lambda_{W}}\mathbf{E} \left [ \|\mathbf{x}^{k}-\mathbf{x}^{\star}\|_{\Lambda_{M}^{2}}^{2} \right],\notag 
\end{align}
leading to
	\begin{align}
		& \mathbf{E} \left[\| \mathbf{z}^k - \mathbf{z}^{\star}\|_Q^2\right] \notag\displaybreak[0]\\
		\leq & \frac{2(1+c_{1})}{\lambda_{W}}\mathbf{E} \left[\|\mathbf{x}^{k+1}-\mathbf{x}^{k}\|_{\left(\Lambda_{M}+D\right)^{2} }^{2}\right] +  \frac{2(1+{c_{1}})}{\lambda_{W}} N \tau \sigma^2\notag\displaybreak[0]\\
		&+ \beta^{2}(1+\frac{1}{c_{1}})(1+c_{2})\mathbf{E} \left[\|\mathbf{x}^{k}\|_{W}^{2} \right] \notag\displaybreak[0]\\
		&+\mathbf{E} \left[ \|\mathbf{x}^{k}-\mathbf{x}^{\star}\|_{\left(\beta R+\frac{\left(1+1 / c_{1}\right)\left(1+1 / c_{2}\right) \Lambda_{M}^{2}}{\lambda_{W}}\right)}^2\right]. \label{bound z}
	\end{align}
Pick an arbitrary $\delta_s \in (0,1)$. By subtracting \eqref{bound z} multiplied by $\delta_s$ from \eqref{penultimate}, we have
	\begin{align}
		&(1-\delta_s)\mathbf{E}  \left[\left\|\mathbf{z}^{k}-\mathbf{z}^{\star}\right\|_{Q}^{2} \right]-\mathbf{E} \left[\left\|\mathbf{z}^{k+1}-\mathbf{z}^{\star}\right\|_{Q}^{2} \right] \notag\displaybreak[0]\\
		&\geq\beta \mathbf{E} \left[\left\|\mathbf{x}^{k}-\mathbf{x}^{\star}\right\|_ {\mathbf{\Omega}_1}^{2}\right] + \Omega_2\mathbf{E} \left[\left\|\mathbf{x}^{k}\right\|_{W}^{2}\right]\notag\displaybreak[0]\\
		&-\mathbf{E} \left[\left\|\mathbf{x}^{k+1}-\mathbf{x}^{k}\right\|_{\mathbf{\Omega}_3}^{2}\right]-(\frac{2 \left(1+c_{1}\right)\delta_s }{\lambda_{W}} + 2 )N\tau \sigma^2,\label{pre-theorem 1}
	\end{align}
where $\mathbf{\Omega}_1 =\left(2 \eta_s m_{\beta} - c_0\right) I_{N d}-\delta_s(R + \frac{\left(1+ 1/c_{1}\right)\left(1+1 / c_{2}\right) \Lambda_{M}^{2}}{\beta \lambda_{W}})$, $\Omega_2 =\beta^{2}(1-\eta_s)-\delta_s \beta^{2}\left(1+1 / c_{1}\right)\left(1+c_{2}\right) $, and $\mathbf{\Omega}_3 = \frac{2\left(1+c_{1}\right)\delta_s}{\lambda_{W}}\left(\Lambda_{M}+D\right)^{2}+ \beta\left(\mathcal{A}_{c_0, \eta_s}+\beta W-R\right)$.
To make \eqref{rate} hold based on \eqref{pre-theorem 1}, it suffices to let $\mathbf{\Omega}_1 \succeq \mathbf{O}_{Nd}$, $\Omega_2 \geq 0$, $\mathbf{\Omega}_3 \preceq \mathbf{O}_{Nd}$, i.e.,
\begin{align}
  \delta_s \leq & \frac{2 \eta_s m_{\beta}-c_0}{\lambda_{\max }\left(R+\left(1+1 / c_{1}\right)\left(1+1 / c_{2}\right) \Lambda_{M}^{2} /\left(\beta \lambda_{W}\right)\right)}, \label{delta 1}\displaybreak[0]\\
	\delta_s \leq & \frac{1-\eta_s}{\left(1+1 / c_{1}\right)\left(1+c_{2}\right)}, \label{delta 2}\displaybreak[0]\\
	\delta_s \leq & \frac{\beta \lambda_{W} \kappa_{c_0, \eta_s}}{2\left(1+c_{1}\right)\left\|\Lambda_{M}+D\right\|^{2}}\label{delta 3}.
\end{align}
To guarantee the existence of $\delta_s \in (0,1)$ subject to \eqref{delta 1}--\eqref{delta 3}, we need $c_0<2\eta_s m_{\beta}$ and $\kappa_{ c_0, \eta_s }>0$. Note from \eqref{D} that $\kappa_{c_0,\eta_s}>0$ at $c_0=2\eta_s m_{\beta}$. Then, due to the continuity of $\kappa_{c_0, \eta_s}$ with respect to $c_0$, there is $c_0\in(0,2\eta_s m_{\beta})$ such that $\kappa_{c_0,\eta_s}>0$. Therefore, we ensure \eqref{rate} with $\delta_s$ given by \eqref{delta_s}. 

Finally, from \eqref{rate}, we have $\mathbf{E}\left[\|\mathbf{z}^{k}-\mathbf{z}^{\star}\|_{Q}^{2}\right]\le(1-\delta_s)^k\mathbf{E}\left[\|\mathbf{z}^{0}-\mathbf{z}^{\star}\|_{Q}^{2}\right]+\sum_{t=0}^{k-1}(1-\delta_s)^t\Gamma N \tau \sigma^2$, $\forall k\ge0$. It can thus be shown that \eqref{neighborhood} holds.

\bibliographystyle{IEEEtran}
\bibliography{reference}
\end{document}